\documentclass[12pt,twoside]{article}
\usepackage{amsfonts}
\usepackage{cite}
 \usepackage{color}
\usepackage{amsmath} \numberwithin{equation}{section}
\usepackage{amsthm}
\topmargin by-\headsep  \textheight 25cm
\newtheorem{lemma}{Lemma}[section]

\newtheorem{corollary}[lemma]{Corollary}
\newtheorem{definition}{Definition}[section]
\newtheorem{theorem}[definition]{Theorem}

\begin{document}
\title{Boundary Layer Study for an Ocean Related  System with a  Small Viscosity Parameter\\}
\author{Wenshu Zhou$^a$\quad Xulong Qin$^b$ \quad Xiaodan Wei$^c$\thanks{Corresponding author. E-mail address:   weixiaodancat@126.com}\quad Xu Zhao$^{d,a}$ \\
     \small a. Department of Mathematics, Dalian Minzu University,  Dalian 116600, P. R. China\\
   \small b. Department of Mathematics, Sun Yat-sen University,  Guangzhou 510275, P. R. China\\
   \small c. School of Computer Science, Dalian Minzu University,  Dalian 116600, P. R. China\\
   \small d. School of Mathematics, Beifang Minzu University,  Yinchuan 750021, P. R. China\\}


\date{}
 \maketitle

\begin{abstract}
\small{We study an ocean related  system with a small viscosity parameter, which is the linearized version of the modified Primitive Equations.  As the parameter goes to zero, a  $L^\infty$ convergence result   is obtained together with the estimation on the thickness of the boundary layer. }

\medskip
 {\bf Keywords}.  Primitive Equations, $L^\infty$ convergence, boundary layer.

\medskip
 {\bf 2010 MSC}. 35L50; 47D06; 76D10; 76N20; 86A05.
\end{abstract}

%
\bigbreak

 \section{Introduction}
We consider the  ocean related  system with a small parameter $\varepsilon$ in the space-time domain $(0, L)\times (0, T)$:
 \begin{equation}\label{equ1}
  \left\{\begin{split}
 &u_t^\varepsilon+{\overline U_0} u^\varepsilon_x+\psi^\varepsilon_x-2\varepsilon u^\varepsilon_{xx}=f,\\
 &\psi^\varepsilon_t+{\overline U_0} \psi^\varepsilon_x+ \lambda^{-2}u^\varepsilon_x=g,
 \end{split}\right.
 \end{equation}
   with the boundary and initial conditions:
\begin{equation}\label{initial}
\left\{\begin{split}&u^\varepsilon(0,t)=u^\varepsilon(L,t)=0,\quad \psi^\varepsilon(0,t)=0,\quad 0<t<T,\\
 &u^\varepsilon(x,0)=u_0(x),\quad \psi^\varepsilon(x,0)=\psi_0(x),\quad 0<x<L,
 \end{split}\right.
 \end{equation} where $0<\varepsilon \ll1, {\overline U_0} >0$ and $\lambda>0$ are constants with ${\overline U_0} < \lambda^{-1}$.  This system was derived from the Primitive Equations  (PEs) of the ocean with mild viscosity
thanks to a modal decomposition in the vertical direction, see \cite{TT2}( cf. \cite{RTT1,RTT2}).

 The PEs are one of the fundamental models for geophysical flows
and they are used to describe oceanic and atmospheric dynamics, see \cite{P,WP}.  In the presence of viscosity, the study of the PEs through analytical means was started by Lions, Temam  and Wang in \cite{LTW1,LTW2}, and some recent advances for the PEs have been obtained, see \cite{CT1} and the  survey \cite{LT2}. In the absence of viscosity,  it is known that the PEs are not well-posed for any set of local boundary conditions, see \cite{OS}. To overcome this difficulty, a modified PEs  (i.e. $\delta$-PEs) was proposed in  \cite{TT2}. System \eqref{equ1} is the linearized version of the modified PEs.

  Formally setting  $\varepsilon=0$ in \eqref{equ1}, one obtains the following system in the space-time domain $(0, L)\times (0, T)$:
 \begin{equation}\label{l1}
  \left\{\begin{split}
 &u^0_t +{\overline U_0} u^0_x+\psi^0_x=f,\\
 &\psi^0_t+{\overline U_0} \psi^0_x+\lambda^{-2}u^0_x=g.
 \end{split}\right.
 \end{equation}
Like in \cite{RTT1,RTT2}, the  boundary and initial conditions are  chosen as follows:
\begin{equation}\label{linitial}
\left\{\begin{split}&u^0(0,t)+{\overline U_0} \lambda^2\psi^0(0,t)=0,\quad u^0(L,t)=0, \quad 0<t<T,\\
 &u^0(x,0)=u_0(x),\quad \psi^0(x,0)=\psi_0(x),\quad 0<x<L.
 \end{split}\right.
 \end{equation}

In \cite{RTT1},  the existence and uniqueness of solutions of both problem \eqref{equ1}-\eqref{initial} and problem \eqref{l1}-\eqref{linitial} was proved for some initial data.  As numerically shown in \cite{RTT2}, some boundary layers appear at the boundary $x=0$  when $\varepsilon$ goes to zero. Actually, it was shown in \cite[Theorem 1.3]{RTT1} that $U^\varepsilon$ and $\Psi^\varepsilon$ are $O(\varepsilon^{1/2})$ in $L^\infty(0,T; L^2(0, L))$, where
  $$
  U^\varepsilon=u^\varepsilon-u^0-\theta_u^\varepsilon,\quad \Psi^\varepsilon=\psi^\varepsilon-\psi^0-\theta_\psi^\varepsilon,
  $$
   where $\theta_u^\varepsilon$ and $\theta_\psi^\varepsilon$ are the boundary layer correctors  and are the same as in Section 2. This implies that $(u^\varepsilon, \psi^\varepsilon)$ converges to $(u^0, \psi^0)$ in $L^\infty(0,T; L^2(0, L))$ as $\varepsilon\rightarrow 0^+$. Recently,  the analysis of the boundary layers for the linearized viscous PEs has been presented in \cite{HJT1} for 2D and in \cite{HJT2,HJT3,HJT4} for 3D.

  In the present paper, we study the $L^\infty$ convergence of  $(U^\varepsilon,\Psi^\varepsilon)$ as $\varepsilon\rightarrow 0^+$. Our main result is as follows.

\begin{theorem}  Assume that $(u^\varepsilon,\psi^\varepsilon) \in \mathcal{K}$ and $(u^0,\psi^0) \in \mathcal{K}^0$ are solutions of problem \eqref{equ1}-\eqref{initial} and problem \eqref{l1}-\eqref{linitial},  respectively, where $\mathcal{K}$ and $\mathcal{K}^0$ are the same as in Section 3. Then there exists some constant $C$ independent of $\varepsilon$ such that
 \begin{equation*}\label{e019}
   \begin{split}&\|(U^\varepsilon, \Psi^\varepsilon)\|_{L^\infty(0, T;L^2(0,L))}+\varepsilon^{1/2}\|(U_x^\varepsilon,\Psi_x^\varepsilon)\|_{L^2((0, L)\times(0, T))}\leq C\varepsilon,\\
          &\|(U^\varepsilon, \Psi^\varepsilon)\|_{L^\infty(0, T;L^\infty(0,L))}  \leq C\varepsilon^{1/2}.
 \end{split}
 \end{equation*}
 \end{theorem}

 From Theorem 1.1 and the definitions of $\theta_u^\varepsilon$ and $\theta_\psi^\varepsilon$  (see Section 2), we immediately obtain
\begin{corollary} Under the assumptions of Theorem 1.1, we have
$ \lim_{\varepsilon\rightarrow 0^+}\|(u^\varepsilon-u^0,\psi^\varepsilon-\psi^0)\|_{L^\infty(0, T;L^\infty(\delta(\varepsilon),L))}=0 $
 for any nonnegative function $\delta(\varepsilon)$ satisfying $\delta(\varepsilon)\rightarrow 0$ and $ {\delta(\varepsilon)}/{\varepsilon}\rightarrow +\infty$ as $\varepsilon\rightarrow 0^+$, and
  $\mathop{\underline\lim}_{\varepsilon\rightarrow 0^+}\|(u^\varepsilon-u^0,\psi^\varepsilon-\psi^0)\|_{L^\infty(0, T;L^\infty(0, \delta(\varepsilon)))}>0 $  for any nonnegative function $\delta(\varepsilon)$ satisfying $\delta(\varepsilon)\rightarrow 0$ and ${\delta(\varepsilon)}/{\varepsilon}\rightarrow c~(\hbox{nonnegative constant})$ as $\varepsilon\rightarrow 0^+$ whenever $(u^0(0,t),\psi^0(0,t)) \not\equiv (0, 0)$ in $(0, T)$. This implies that the boundary layer thickness is of the order $O(\varepsilon)$.
 \end{corollary}

  Theorem 1.1 will be proved in Section 2. The main difficulty in the proof is that system \eqref{equ1} is an incompletely parabolic perturbation of a hyperbolic system. To overcome the difficulty, a key observation is to obtain the equation  \eqref{e3}. With this, one can deduce  the required estimates, for example, \eqref{e15}. In Section 3, some remarks on the regularities of  $(u^\varepsilon,\psi^\varepsilon) $ and $(u^0,\psi^0)$ will be   presented.
The regularity conditions   will be used to show  the equation \eqref{e222}.

 \section{Proof of Theorem 1.1}
 Let $(u^0,\psi^0)\in \mathcal{K}^0$ be a solution of   problem \eqref{l1}-\eqref{linitial}. As in \cite{RTT1},   $\theta_u^\varepsilon$ and $\theta_\psi^\varepsilon$ are defined as follows:
 \begin{equation}\label{e1}
 \left( \begin{split}&\theta_u^\varepsilon(x,t)\\
 &\theta_\psi^\varepsilon(x,t)
 \end{split} \right)=e^{-rx/\varepsilon}\left( \begin{split}&A^\varepsilon(t)\\
 &B^\varepsilon(t)
 \end{split} \right)+\left( \begin{split}&C^\varepsilon(t)\\
 &D^\varepsilon(t)
 \end{split} \right),
 \end{equation}
 where $r=\frac{1}{2{\overline U_0}  \lambda^2}-\frac{{\overline U_0}  }{2}>0$, and
 \begin{equation}\label{equation1}
 \left( \begin{split}&A^\varepsilon\\
 &B^\varepsilon\\
 &C^\varepsilon\\
 &D^\varepsilon
 \end{split} \right)= \left( \begin{split}&-(1-e^{-rL/\varepsilon})^{-1}u^0(0,t)\\
  & {\overline U_0}  ^{-1}\lambda^{-2}(1-e^{-rL/\varepsilon})^{-1}u^0(0,t) \\
 &e^{-rL/\varepsilon}(1-e^{-rL/\varepsilon})^{-1}u^0(0,t)\\
 & -{\overline U_0}  ^{-1}\lambda^{-2}e^{-rL/\varepsilon}(1-e^{-rL/\varepsilon})^{-1}u^0(0,t)\\
 \end{split} \right).
 \end{equation}
 Due to $u^0(0,t)=-\overline U_0\lambda^2\psi^0(0,t)$, it is easy to verify that $\theta_u^\varepsilon$ and $\theta_\psi^\varepsilon$ satisfy
 \begin{equation}\label{e34}
\left\{ \begin{split}
 &{\overline U_0}  \theta_{ux}^\varepsilon+\theta_{\psi x}^\varepsilon-2\varepsilon\theta_{uxx}^\varepsilon=0,\quad
  {\overline U_0}  \theta_{\psi x}^\varepsilon+\lambda^{-2}\theta_{ux}^\varepsilon=0,\\
 &\theta_{u}^\varepsilon(0,t)=-u^0(0,t),\quad\theta_{\psi}^\varepsilon(0,t)=-\psi^0(0,t),\quad\theta_u^\varepsilon(L,t)=0.
 \end{split}\right.
 \end{equation}

From now on,    we use $C$ to denote a positive generic constant independent of $\varepsilon$.
 \begin{lemma}Assume that $(u^0,\psi^0)\in \mathcal{K}^0$ is a solution of   problem \eqref{l1}-\eqref{linitial}. Then
 \begin{equation}\label{equ20}
  \left\{\begin{split}&\|(\theta_u^\varepsilon,\theta_\psi^\varepsilon,\theta_{ut}^\varepsilon,\theta_{\psi t}^\varepsilon)\|_{L^\infty(0, T; L^2(0, L))}\leq C\varepsilon^{1/2}, \\
 &\|(\theta_{utx}^\varepsilon,\theta_{\psi tx}^\varepsilon)\|_{L^\infty(0, T; L^2(0, L))}\leq C,\\
 &\|x(\theta_{ut}^\varepsilon,\theta_{\psi t}^\varepsilon,\varepsilon \theta_{utx},\varepsilon\theta_{\psi tx})\|_{L^\infty(0, T; L^2(0, L))}\leq C\varepsilon^{3/2}.
 \end{split}\right.
 \end{equation}
 \end{lemma}
 \begin{proof}
  Due to $u^0(L,t)=0$,  $u_t^0(0,t) =-\int_0^L u^0_{xt}(x,t)dx$. Thus, $u_t^0(0,t)\in L^\infty(0, T)$ since $u^0_{xt} \in L^\infty(0, T; L^2(0,L))$. This together with $u^0(0,t)=-\overline U_0\lambda^2\psi^0(0,t)$ gives $\psi_t^0(0,t)\in L^\infty(0, T)$. Using the inequality $e^{-x}\leq 6/x^3~(\forall x>0)$, we have $e^{-2rL/\varepsilon}\leq C\varepsilon^3$. Noticing $\|xe^{-rx/\varepsilon}\|_{L^2(0, L)}\leq C\varepsilon^{3/2}$, we deduce that $\|x\theta_{ut}\|_{L^\infty(0,T;L^2(0,L))}\leq C\varepsilon^{3/2}.$ The other estimates of \eqref{equ20} can be deduced similarly. The proof is completed.
 \end{proof}

 \noindent{\bf Proof of Theorem 1.1}
 For simplicity, write $u=U^\varepsilon$ and $\psi=\Psi^\varepsilon$. It follows from  \eqref{equ1}, \eqref{l1} and \eqref{e34} that $u$ and $\psi$ satisfy
 \begin{equation}\label{equation}
  \left\{\begin{split}
 &u_t+{\overline U_0}  u_x+\psi_x-2\varepsilon u_{xx}=2\varepsilon u^0_{xx}-\theta_{ut}^\varepsilon,\\
 &\psi_t+{\overline U_0}  \psi_x+ \lambda^{-2}u_x=-\theta_{\psi t}^\varepsilon,
 \end{split}\right.
 \end{equation}
with the boundary and initial  conditions
 \begin{equation}\label{e21}
 \left\{
 \begin{split}&u(0,t)=u(L,t)=0,\quad \psi(0,t)=0,\quad 0<t<T,\\
 &u(x,0)=0,\quad \psi(x,0)=0,\quad 0<x<L.
 \end{split}\right.
 \end{equation}
 From  \eqref{equation}, we obtain
 \begin{equation}\label{e22}
   \begin{split}
   &(u-\lambda\psi)_t+{\overline U_0}  (u-\lambda\psi)_x- \lambda^{-1}(u-\lambda\psi)_x-2\varepsilon u_{xx}=2\varepsilon u^0_{xx}-\theta_{ut}^\varepsilon+\lambda\theta_{\psi t}^\varepsilon.
 \end{split}
 \end{equation}
 Differentiating  \eqref{equation}$_2$ in $x$ and multiplying the resulting equation by $2\varepsilon\lambda^2$, we have
 \begin{equation}\label{e222}
   \begin{split}
    2\varepsilon\lambda^2\psi_{xt}+2\varepsilon {\overline U_0}  \lambda^2\psi_{xx}+2\varepsilon u_{xx}=-2\varepsilon\lambda^2\theta_{\psi tx}^\varepsilon.
 \end{split}
 \end{equation}
Denote $W=u-\lambda\psi+2\varepsilon\lambda^2\psi_x$.  Adding \eqref{e22} and \eqref{e222} yields
\begin{equation}\label{e3}
   \begin{split}
   &W_t+{\overline U_0}  W_x-\lambda^{-1}(u-\lambda\psi)_x=2\varepsilon u^0_{xx}-\theta_{ut}^\varepsilon+\lambda\theta_{\psi t}^\varepsilon-2\varepsilon\lambda^2\theta_{\psi tx}^\varepsilon=:2\varepsilon u^0_{xx}+F^\varepsilon.
 \end{split}
 \end{equation}
Multiplying \eqref{e3} by $W$, integrating over $(0, L)\times (0, t)$ and using \eqref{e21}, we have
 \begin{equation}\label{e4}
   \begin{split}
   &\frac12\int_0^LW^2dx+\frac{{\overline U_0}  }{2}\int_0^tW^2|_{x=L}ds+2\varepsilon\lambda^2\int_0^t\hspace{-1mm}\int_0^L\psi_x^2dxds \\
   &=\frac{1}{2}\int_0^LW^2|_{t=0}dx+\frac{{\overline U_0}  }{2}\int_0^tW^2|_{x=0}ds+\frac{1}{2\lambda}\int_0^t(u-\lambda\psi)^2|_{x=0}^{x=L}ds\\
    &\quad+2\varepsilon\lambda\int_0^t\hspace{-1mm}\int_0^Lu_x\psi_xdxds
   +\int_0^t\hspace{-1mm}\int_0^L(2\varepsilon u^0_{xx}+F^\varepsilon)Wdxds\\
   &=\frac{\lambda}{2}\int_0^t\psi^2|_{x=L}ds+2\varepsilon^2{\overline U_0}  \lambda^4\int_0^t\psi_x^2|_{x=0}ds+2\varepsilon\lambda\int_0^t\hspace{-1mm}\int_0^Lu_x\psi_xdxds \\
      &\quad+\int_0^t\hspace{-1mm}\int_0^L 2\varepsilon u^0_{xx} Wdxds +\int_0^t\hspace{-1mm}\int_0^L F^\varepsilon (u-\lambda\psi)dxds +2\varepsilon\lambda^2\int_0^t\hspace{-1mm}\int_0^L F^\varepsilon \psi_x dxds \\
   &=:\frac{\lambda}{2}\int_0^t\psi^2|_{x=L}ds+\sum\limits_{i=1}^5E_i.
 \end{split}
 \end{equation}
Using the Hardy inequality (cf. \cite{Hardy}) and noticing \eqref{e21}, we obtain
\begin{equation}\label{hardy3}
   \begin{split}
  \int_0^L\frac{u^2}{x^2}dx\leq 4\int_0^Lu_x^2dx,\quad \int_0^L\frac{\psi^2}{x^2}dx\leq 4\int_0^L\psi_x^2dx.
 \end{split}
 \end{equation}
  By Lemma 2.1, we have
 \begin{equation}\label{hardy1}
   \begin{split}
  \|F^\varepsilon\|_{L^\infty(0, T;L^2(0,L))}^2\leq C\varepsilon,\quad \|xF^\varepsilon\|_{L^\infty(0, T;L^2(0,L))}^2\leq C\varepsilon^3.\\
 \end{split}
 \end{equation}
 Applying the Young inequality, \eqref{hardy1} and \eqref{hardy3}, we deduce that
  \begin{equation}\label{hardy2}
   \begin{split}
  E_4 \leq& \frac{C}{\varepsilon}\|xF^\varepsilon\|_{L^2((0, T)\times (0, L))}^2+\frac{\varepsilon}{32}\int_0^t\hspace{-1mm}\int_0^L\frac{(u-\lambda\psi)^2}{x^2}dxds\\
  \leq & C\varepsilon^2+\frac{\varepsilon}{4}\int_0^t\hspace{-1mm}\int_0^L(u_x^2+\lambda^2\psi_x^2)dxds,
 \end{split}
 \end{equation}
 and
\begin{equation}\label{e7}
   \begin{split}
   E_5
  \leq&C\varepsilon^2+\frac{\varepsilon\lambda^2}{4}\int_0^t\hspace{-1mm}\int_0^L\psi_{x}^2dxds.
 \end{split}
 \end{equation}
  Using \eqref{equation}$_2$, we have $\psi_x|_{x=0}=-   {\overline U_0}^{-1}  \lambda^{-2} u_x|_{x=0}- {\overline U_0}^{-1}\theta^\varepsilon_{\psi t}|_{x=0}$. Noticing $\psi_t^0(0,t)\in L^\infty(0, T)$ (see the proof of Lemma 2.1), we have
  \begin{equation}\label{equ22}
   \begin{split}
  E_1 \leq & C\varepsilon^2+\frac{4\varepsilon^2}{{\overline U_0}  }\int_0^tu_x^2(0, s)ds.
 \end{split}
 \end{equation}
 Since $u(0,t)=u(L,t)=0$, there exists some $\xi=\xi_t \in (0, L)$ such that $u_x(\xi,t)=0$, thus,
  $   u_x^2(0,t)=-\int_0^\xi (u_x^2)_x dx.  $
 Substituting it into \eqref{equ22} and using the Young inequality, we obtain
  \begin{equation}\label{e5}
   \begin{split}
  E_1
  \leq&C\varepsilon^2+\frac{8\varepsilon^2}{{\overline U_0}  }\int_0^t\hspace{-1mm}\int_0^L|u_xu_{xx}|dxds\\\
  \leq&C\varepsilon^2+\frac{64\varepsilon}{\overline U_0^2\lambda^2}\int_0^t\hspace{-1mm}\int_0^L u_x^2dxds+\frac{\varepsilon^3\lambda^2}{4}\int_0^t\hspace{-1mm}\int_0^Lu_{xx}^2dxds.\\
 \end{split}
 \end{equation}
  Using   $u^0_{xx}\in L^\infty(0,T; L^2(0, L))$ and Lemma 2.1, we derive from  \eqref{equation}$_1$ that
   \begin{equation}\label{e6}
   \begin{split}
    &\int_0^t\hspace{-1mm}\int_0^L(4\varepsilon^2u_{xx}^2+u_t^2)dxds+2\varepsilon\int_0^Lu_{x}^2dx\\
    &= \int_0^t\hspace{-1mm}\int_0^L (u_t- 2\varepsilon u_{xx})^2dxds \\
    &=\int_0^t\hspace{-1mm}\int_0^L (2\varepsilon u^0_{xx}-\theta_{ut}^\varepsilon-{\overline U_0}  u_x-\psi_x)^2dxds\\
     &\leq C\varepsilon+4\overline U_0^2\int_0^t\hspace{-1mm}\int_0^L u_x^2dxds+4\int_0^t\hspace{-1mm}\int_0^L \psi_x^2dxds,
 \end{split}
 \end{equation}
 consequently,
  \begin{equation}\label{equ19}
   \begin{split}
    &\varepsilon^2\int_0^t\hspace{-1mm}\int_0^L u_{xx}^2 dxds \leq C\varepsilon+ \overline U_0^2\int_0^t\hspace{-1mm}\int_0^L u_x^2dxds+ \int_0^t\hspace{-1mm}\int_0^L \psi_x^2dxds.\\
 \end{split}
 \end{equation}
 Substituting \eqref{equ19} into \eqref{e5} yields
 \begin{equation}\label{e8}
   \begin{split}
   E_1
  \leq&C\varepsilon^2+\Big(\frac{64}{ \overline U_0^2\lambda^2}+\frac{ \overline U_0^2\lambda^2}{4}\Big)\varepsilon \int_0^t\hspace{-1mm}\int_0^L u_x^2dxds+\frac{\varepsilon\lambda^2}{4}\int_0^t\hspace{-1mm}\int_0^L\psi_{x}^2dxds.\\
 \end{split}
 \end{equation}
 By the Young inequality, we have
 \begin{equation}\label{e11}
   \begin{split}
   E_2 +E_3
  \leq& C\varepsilon^2+\frac12\int_0^t\hspace{-1mm}\int_0^LW^2dxds+4\varepsilon\int_0^t\hspace{-1mm}\int_0^Lu_{x}^2dxds+\frac{\varepsilon\lambda^2}{4}\int_0^t\hspace{-1mm}\int_0^L\psi_{x}^2dxds.\\
 \end{split}
 \end{equation}
  Plugging \eqref{hardy2}, \eqref{e7}, \eqref{e8} and \eqref{e11} into \eqref{e4}, we obtain
  \begin{equation}\label{e9}
   \begin{split}
   &\frac12\int_0^LW^2dx+ \varepsilon\lambda^2\int_0^t\hspace{-1mm}\int_0^L\psi_x^2dxds\\
   &\leq C\varepsilon^2+\frac{\lambda}{2}\int_0^t\psi^2|_{x=L}ds +C_0\varepsilon\int_0^t\hspace{-1mm}\int_0^L u_x^2dxds+\frac12\int_0^t\hspace{-1mm}\int_0^L  W^2dxds,
 \end{split}
 \end{equation}
 where $C_0=\frac{17}{4}+\frac{64}{\overline U_0^2\lambda^2}+\frac{\overline U_0^2\lambda^2}{4}$.

 On the other hand, multiplying   \eqref{equation}$_1$ and \eqref{equation}$_2$ by $2u$ and $2\lambda^2\psi$  respectively,  integrating over $(0, L)\times (0, t)$,  using \eqref{e21}  and \eqref{hardy3}, and performing a similar argument to \eqref{hardy2}, we obtain
  \begin{equation}\label{e10}
   \begin{split}
   & \int_0^L(u^2+\lambda^2\psi^2)dx+{\overline U_0}\lambda^2\int_0^t\psi^2|_{x=L} ds+4\varepsilon\int_0^t\hspace{-1mm}\int_0^L u_x^2dxds\\
     &= 2\int_0^t\hspace{-1mm}\int_0^L\left(  2\varepsilon u^0_{xx}u-\theta_{ut}^\varepsilon u-\lambda^2\theta_{\psi t}^\varepsilon\psi\right)dxds\\
& \leq C\varepsilon^2+ C\int_0^t\hspace{-1mm}\int_0^L u^2 dxds+2\varepsilon\int_0^t\hspace{-1mm}\int_0^L u_x^2 dxds+ \frac{\varepsilon\lambda^2}{2(C_0+1/(2\lambda{\overline U_0}))}\int_0^t\hspace{-1mm}\int_0^L \psi_x^2dxds.\\
 \end{split}
 \end{equation}
Hence,
\begin{equation}\label{e13}
   \begin{split}
   & \int_0^L(u^2+\lambda^2\psi^2)dx+{\overline U_0}\lambda^2\int_0^t\psi^2|_{x=L} ds+ 2\varepsilon\int_0^t\hspace{-1mm}\int_0^L u_x^2dxds \\
& \leq C\varepsilon^2+ C\int_0^t\hspace{-1mm}\int_0^L u^2 dxds+  \frac{\varepsilon\lambda^2}{2(C_0+1/{(2\lambda\overline U_0)}  )}\int_0^t\hspace{-1mm}\int_0^L \psi_x^2dxds.
 \end{split}
 \end{equation}
 Multiplying \eqref{e13} by $C_0+1/(2\lambda{\overline U_0})$  and adding the resulting equations to \eqref{e9}, we have
 \begin{equation*}\label{e14}
   \begin{split}
   & \frac12\int_0^LW^2dx+C_0\int_0^L(u^2+\lambda^2\psi^2)dx + \varepsilon\int_0^t\hspace{-1mm}\int_0^L\left(\frac{\lambda^2}{2}\psi_x^2+C_0u_x^2\right)dxds \\
   &\leq C\varepsilon^2+ C\int_0^t\hspace{-1mm}\int_0^L (u^2+W^2) dxds.\
 \end{split}
 \end{equation*}
 Then, the Gronwall inequality gives
 \begin{equation}\label{e14}
   \begin{split}
   & \sup\limits_{0<t<T}\int_0^L(W^2+u^2+\psi^2)dx + \varepsilon\int_0^T\hspace{-1mm}\int_0^L\left(\psi_x^2+u_x^2\right)dxds \leq C\varepsilon^2.
 \end{split}
 \end{equation}
 Recalling \eqref{e6} and $W=u-\lambda\psi+2\varepsilon\lambda^2\psi_x$ and using \eqref{e14}, we  deduce that
 \begin{equation}\label{e15}
   \begin{split}
    \sup\limits_{0<t<T}\int_0^L(u_x^2+\psi_x^2)dx\leq C.
 \end{split}
 \end{equation}
Thanks to $u(0,t)=0$, we have $u^2(x,t) = \int_0^x (u^2)_xdx$. Then, using the H\"{o}lder inequality, \eqref{e14} and \eqref{e15}, we obtain
 \begin{equation*}\label{e17}
   \begin{split}
       u^2(x,t)
       \leq  2\left(\int_0^Lu^2dx\int_0^Lu_x^2dx\right)^{1/2} \leq C\varepsilon,
 \end{split}
 \end{equation*}
thus, $       \|U^\varepsilon\|_{L^\infty(0, T;L^\infty(0,L))}  \leq C\varepsilon^{1/2}. $
 Similarly,  $       \|\Psi^\varepsilon\|_{L^\infty(0, T;L^\infty(0,L))}  \leq C\varepsilon^{1/2}. $
  The proof of Theorem 1.1 is completed.

 \section{Remarks on Regularity of Solutions}
By the Hille-Yosida theorem, the authors in \cite{RTT1} proved the following results on existence and uniqueness of both problem \eqref{equ1}-\eqref{initial} and problem \eqref{l1}-\eqref{linitial}:

  (i)~~If $(f, g) \in L^1(0, T; H)$, $(u_0,\psi_0)\in D(A)$, and $(f, g)$ is continuous in $H$ at $t=0$, then for every $\varepsilon>0 $ problem \eqref{equ1}-\eqref{initial} admits a unique solution $(u^\varepsilon, \psi^\varepsilon)$    in ${\mathcal F}=C([0, T]; H)\cap L^\infty(0, T; D(A))$ with $(u^\varepsilon_t, \psi^\varepsilon_t) \in L^\infty(0, T; H)$, where  $H=L^2(0, L)\times L^2(0, L)$ and
  $
  D(A)=\{(u,\psi)\in H| u_x,\psi_x,u_{xx}\in L^2(0, L), u(0)=\psi(0)=u(L)=0\}.
  $

  (ii)~~If $(f, g) \in L^1(0, T; H)$ and $(u_0,\psi_0)\in D(A^0)$, then   problem \eqref{equ1}-\eqref{initial} admits a unique solution $(u^0, \psi^0)$ in ${\cal F}^0=C([0, T]; H)\cap L^\infty(0, T; D(A^0))$ with $(u^0_t, \psi^0_t) \in L^\infty(0, T; H)$,   where
  $
  D(A^0)=\{(u,\psi)\in H| u_x,\psi_x \in L^2(0, L), u(0)+{\overline U_0} \lambda^2\psi(0)=0,  u(L)=0\}.
  $

    To get $(u^\varepsilon, \psi^\varepsilon) \in \mathcal{K}$, where
    $$\mathcal{K}=\{(u, \psi)\in {\cal F}| \psi_{xt},\psi_{xx}\in L^\infty(0,T;L^2(0,L)), (u_t, \psi_t) \in C([0, T]; H)\},
   $$
   some additional conditions on $ f, g, u_0, \psi_0$ must be imposed, for example,  the following conditions: $(f_t, g_t) \in L^1(0, T; H)$, $(u_{0},\psi_{0})\in D(A)\cap (H^4(0,L)\times H^3(0,L))$,   $(f_t, g_t)$ is continuous in $H$ at $t=0$, and
  \begin{equation}\label{e33}
 \left\{\begin{split}
 &u_{0xx}(0)=u_{0xx}(L)=0,\\
 &{\overline U_0} u_{0x}(0)+\psi_{0x}(0)=f(0,0),\quad {\overline U_0} u_{0x}(L)+\psi_{0x}(L)=f(L,0),\\
 & {\overline U_0} \psi_{0x}(0)+\lambda^{-2}u_{0x}(0)=g(0,0).
 \end{split} \right.
 \end{equation}
     Indeed, we observe  by differentiating the equations in \eqref{equ1} with respect to $t$ that $(u^\varepsilon_t, \psi^\varepsilon_t)$ satisfies  \eqref{equ1} with $(f_t, g_t)$ instead of $(f, g)$ and the initial condition $
  (u^\varepsilon_t|_{t=0}, \psi^\varepsilon_t|_{t=0})=(f(x,0)-{\overline U_0}  u_{0x}-\psi_{0x}+2\varepsilon u_{0xx},~~ g(x,0)-{\overline U_0}  \psi_{0x}-\lambda^{-2}u_{0x}).
  $
 From  the above assumptions, we find that   $(u^\varepsilon_t|_{t=0}, \psi^\varepsilon_t|_{t=0}) \in D(A)$.
Thus, by a similar argument to (i),  one has $(u^\varepsilon_t, \psi^\varepsilon_t) \in C([0, T]; H)\cap L^\infty(0, T; D(A))$  and then, one deduces from \eqref{equ1}$_2$  that $\psi^\varepsilon_{x} \in L^\infty(0,T;H^1(0, L))$ if $g_x \in L^\infty(0,T;L^2(0,L))$. Consequently, $(u^\varepsilon, \psi^\varepsilon) \in \mathcal{K}$.

   To get $(u^0,\psi^0)\in \mathcal{K}^0$, where
     $$\mathcal{K}^0=\{(u,\psi)\in {\cal F}^0| \psi_{xt}, u_{xx}, \psi_{xx}\in L^\infty(0, T; L^2(0,L)), (u_{t}, \psi_{t} ) \in C([0, T]; H)\},$$
     we observe  that under \eqref{e33}, $ f, g, u_0, \psi_0$ satisfy the compatibility conditions (1.72) of \cite{RTT1}. Consequently, if in addition we assume that $(f_x, g_x) \in L^\infty(0, T; H)$ and $u_{0xx},\psi_{0xx},f_t|_{t=0},g_t|_{t=0} \in L^2(0,L)$,
  then $(u^0,\psi^0) \in \mathcal{K}^0$, see \cite[Remark 1.4]{RTT1} for the detail.

 \section*{Acknowledgments}
  The research  was  supported  by  the NSFC (11571062, 11571380),  the Program for Liaoning
Innovative Talents in University (LR2016004),  Guangzhou Science and Technology Program (201607010144)   and  the Fundamental Research Fund  for the Central Universities (DMU).

\par

\end{document}